\documentclass[10pt, a4paper]{amsart}
\usepackage{amsfonts}
\usepackage{amsthm}
\usepackage{amsmath}
\usepackage{longtable}
\usepackage{amssymb}
\usepackage{latexsym}

\theoremstyle{plain} \newtheorem{lem}{Lemma}[section]
\theoremstyle{plain} \newtheorem{prop}[lem]{Proposition}
\theoremstyle{plain} 
\theoremstyle{plain} 

\theoremstyle{definition}

\newtheorem{rem}[lem]{Remark}

\newtheorem{exa}[lem]{Example}

\numberwithin{equation}{section}

\newcommand{\C}{\mathbb{C}}
\newcommand{\Q}{\mathbb{Q}}
\newcommand{\Z}{\mathbb{Z}}
\newcommand{\R}{\mathbb{R}}
\newcommand{\A}{\mathcal{A}}

\newcommand{\GL}{\mathrm{\mathop{GL}}}
\newcommand{\End}{\mathrm{\mathop{End}}}

\newcommand{\ad}{\mathrm{\mathop{ad}}}

\newcommand{\Aut}{\mathrm{\mathop{Aut}}}

\newcommand{\gl}{\mathfrak{\mathop{gl}}}

\newcommand{\ssl}{\mathfrak{\mathop{sl}}}
\newcommand{\su}{\mathfrak{\mathop{su}}}
\newcommand{\so}{\mathfrak{\mathop{so}}}
\newcommand{\ssp}{\mathfrak{\mathop{sp}}}
\newcommand{\g}{\mathfrak{g}}

\newcommand{\hh}{\mathfrak{h}}
\newcommand{\uu}{\mathfrak{u}}

\newcommand{\kk}{\mathfrak{k}}
\newcommand{\pp}{\mathfrak{p}}

\newcommand{\veps}{\varepsilon}

\begin{document}

\title{Constructing semisimple subalgebras of real semisimple Lie algebras}

\author{Paolo Faccin}
\address{Dipartimento di Matematica\\
Universit\`{a} di Trento\\
Italy}
\email{faccin@science.unitn.it}

\author{Willem A. de Graaf}
\address{Dipartimento di Matematica\\
Universit\`{a} di Trento\\
Italy}
\email{degraaf@science.unitn.it}

\begin{abstract}
We consider the problem of constructing semisimple subalgebras of real 
(semi-) simple Lie algebras. We develop computational methods that help to deal with 
this problem. Our methods boil down to solving a set of polynomial equations.
In many cases the equations turn out to be sufficiently ``pleasant'' to
be able to solve them. In particular this is the case for $S$-subalgebras.
\end{abstract}

\maketitle

\section{Introduction}\label{sec:1}

The subject of this paper is the problem of finding semisimple subalgebras of real semisimple
Lie algebras. 
The analogous problem for complex Lie algebras has been widely studied 
(see for example \cite{dyn0}, \cite{dyn}, \cite{minchenko}, \cite{graaf_sss}).
In order to describe the main results in this area we need to introduce some terminology.
Let $\tilde\g^c$ be a semisimple complex Lie algebra, with adjoint group
$\widetilde{G}^c$ (this is the group of inner automorphisms). Two subalgebras
$\g_1^c,\g_2^c\subset\tilde\g^c$ are said to be {\em equivalent} if
there is an $\eta\in \widetilde{G}^c$ with $\eta(\g_1^c) = \g_2^c$. They
are called {\em linearly equivalent} if for all representations $\rho : \tilde\g^c \to
\gl(V^c)$ we have that the subalgebras $\rho(\g_1^c)$, $\rho(\g_2^c)$ are
conjugate under $\GL(V^c)$. A subalgebra of $\tilde\g^c$ is called {\em regular} if it
is normalized by a Cartan subalgebra of $\tilde\g^c$. An $S$-subalgebra is a subalgebra
not contained in a regular subalgebra. We have the following:
\begin{itemize}
\item There is an algorithm to determine the regular semisimple subalgebras of $\tilde\g^c$,
up to equivalence (\cite{dyn}).
\item The maximal semisimple $S$-subalgebras of the simple Lie algebras of classical type
(\cite{dyn0}), and the semisimple $S$-subalgebras of the simple Lie algebras of exceptional type
(\cite{dyn}) have been classified up to equivalence.
\item The simple subalgebras of the Lie algebras of exceptional type have been classified
up to equivalence (\cite{minchenko}).
\item The semisimple subalgebras of the simple Lie algebras of ranks not exceeding 8
have been classified up to linear equivalence (\cite{graaf_sss}).
\end{itemize}

Now let $\tilde\g$ be a real semisimple Lie algebra with adjoint group $\widetilde{G}$.
A classification of the semisimple subalgebras of $\tilde\g$, up to $\widetilde{G}$-conjugacy,
appears to be completely out of reach. Therefore we consider a weaker problem. Note that
if $\g\subset \tilde\g$, then also for the compexifications,
$\g^c = \C\otimes \g$, $\tilde\g^c =\C\otimes\tilde\g$ we have that 
$\g^c\subset \tilde \g^c$. So assume that we know an inclusion $\g^c \subset \tilde\g^c$. 
This leads to the following problem:
let $\tilde\g^c$ be a complex semisimple Lie algebra, and $\g^c$ a complex 
semisimple subalgebra of it. Let $\g\subset \g^c$ be a real form of $\g^c$.
List, up to isomorphism, all real forms $\tilde\g \subset \tilde\g^c$ of
$\tilde\g^c$ such that $\g\subset \tilde\g$. 

We recall the following fact (\cite{onishchik}, \S2, Proposition 1): 
let $\tilde\g, \tilde\g'\subset \tilde\g^c$
be two real forms of $\tilde\g^c$. Then $\tilde\g$ and $\tilde\g'$ are 
isomorphic if and only if there is a $\phi\in \Aut(\tilde\g^c)$ such that
$\phi(\tilde\g)= \tilde\g'$.

Because of this we can reformulate the problem as follows: let $\veps : 
\g^c\hookrightarrow \tilde\g^c$ be an embedding of complex semisimple Lie
algebras. Let $\g\subset \g^c$ be a real form. List, up to isomorphism,
all real forms $\tilde\g$ of $\tilde\g^c$ such that there is a 
$\phi\in \Aut(\tilde\g^c)$ with $\phi(\veps(\g))\subset \tilde\g$.
This is the main problem that we consider in this paper. 

Let $\tilde\g_1,\ldots,\tilde\g_m$ be the non-compact real forms of $\tilde\g^c$
(i.e., each non-compact real form of $\tilde\g^c$ is isomorphic to exactly one $\tilde\g_i$).
In our setting the $\tilde\g_i$ are given by a basis and a multiplication table 
(see Section \ref{sec:comp}). In this paper we describe algorithmic methods that help
to solve the following problem: given an embedding $\veps : \g^c\hookrightarrow \tilde\g^c$,
and a real form $\g$ of $\g^c$, find all $i$ such that there is an automorphism $\phi$ of
$\tilde\g^c$ such that $\phi(\veps(\g))\subset \tilde\g_i$, along with a basis of the subalgebra
$\phi(\veps(\g))$ of $\tilde\g_i$ in terms of a basis of $\tilde\g_i$. Our algorithms reduce this
problem to finding the solution to a set of polynomial equations. We show some nontrivial
examples where it is possible to deal with these polynomial equations. Our approach is 
particularly well suited for $S$-subalgebras; at the end of the paper
we give a list of all $\tilde\g_i$, when $\tilde\g^c$ is of exceptional type and
the image of $\veps$ is an $S$-subalgebra of $\tilde\g^c$.

For real semisimple Lie algebras the problem of finding and classifying the semisimple
subalgebras has previously been considered in the literature. 
Cornwell has published a series of papers on this
topic, \cite{cornsub1}, \cite{cornsub2}, \cite{cornsub3}, \cite{cornsub4}, the last two
in collaboration with Ekins. Their methods require detailed case-by-case calculations,
and it is not entirely clear whether they are applicable to every subalgebra. For example,
no $S$-subalgebras are considered in these publications (except for some $S$-subalgebras of type
$A_1$ in the Lie algebras of types $G_2$ and $F_4$).

Komrakov (\cite{komrakov}) classified the maximal proper semisimple Lie subalgebras of a real
simple Lie algebra. However, his paper does not give an account of the methods used.
He also has a list of the real forms which contain a maximal $S$-subalgebra, for $\tilde\g^c$
of exceptional type. We find the same inclusions as Komrakov, except that in type $E_6$
we find a few more (see Section \ref{sec:Ssub}).

Now we give an outline of the paper. The next section contains concepts and constructions
from the literature that we use. We also give an algorithm to compute equivalences of 
representations of semisimple Lie algebras, which may not have been described before,
but follows immediately from the representation theory of such algebras. In Section \ref{sec:3}
we describe our method. Section \ref{sec:4} has some examples computed using our implementation.
Finally, in Section \ref{sec:Ssub} we give the list of real semisimple subalgebras of the
real simple Lie algebras of exceptional type, that correspond to $S$-subalgebras of 
the corresponding complex Lie algebras.

\subsection{Computational set up}\label{sec:comp}

We have implemented the algorithms in the language of the computer algebra system
{\sf GAP}4 (\cite{gap4}), using the package {\sf CoReLG} (\cite{corelg}). In this
system a Lie algebra is given by a basis and a multiplication table. The package
{\sf CoReLG} contains functionality for constructing all real forms of a simple
complex Lie algebra (see \cite{dfg}). So in our implementations we work with 
Lie algebras given in that way. An element of an algebra is represented by its coefficient
vector relative to the given basis of the algebra. Subspaces (in particular, subalgebras)
are given by a basis. Linear maps (in particular, automorphisms) are defined
with respect to the given basis of the Lie algebra. And so on.

Also we use the {\sf GAP}4 package {\sf SLA} (\cite{sla}), which contains a database
of the semisimple subalgebras of the simple complex Lie algebras of ranks not exceeding 8.
We use this database to obtain the starting data for our algorithms: the embeddings
$\veps : \g^c \hookrightarrow \tilde\g^c$.

\subsection{Notation}

Throughout we endow symbols denoting vector spaces or algebras over the complex numbers 
by a superscript $c$. If this superscript is absent, then the vector space, or algebra,
is defined over the reals. In the above discussion we have already used this convention.

We use standard notation and terminology for Lie algebras, as can for instance be found
in the books of Humphreys (\cite{hum}) and Onishchik (\cite{onishchik}). 
Lie algebras will be denoted by fraktur symbols (like $\g$). The adjoint representation
of a Lie algebra $\g$ is defined by $\ad_\g x(y) = [x,y]$. We also just use $\ad$ if no
confusion can arise about which Lie algebra is meant. 

We denote the real forms of the simple Lie algebras using the convention of \cite{knapp02},
Appendix C.3 and C.4, see also \cite{onishchik}, Table 5. 

We denote the imaginary unit in $\C$ by $\imath$.

\section{Preliminaries}

\subsection{Semisimple real Lie algebras}\label{subsec:pre1}

Let $\g^c$ be a semisimple Lie algebra over $\C$. Let $\hh^c$ be a fixed Cartan
subalgebra of $\g^c$, and let $\Phi$ denote the corresponding root system.
By $\Delta = \{\alpha_1,\ldots,\alpha_\ell\}$ we denote a basis of simple roots
of $\Phi$, corresponding to a choice of positive roots $\Phi^+$.
For $\alpha,\beta\in \Phi$ we let $r,q$ be the maximal integers 
such that $\beta-r\alpha$ and $\beta +q\alpha$ lie in $\Phi$, and we define
$\langle \beta,\alpha^\vee\rangle=r-q$. For $\alpha\in \Phi$ we denote by
$\g^c_\alpha$ the corresponding root space in $\g^c$.

There is a basis of $\g^c$ formed by elements $h_1,\ldots,h_\ell\in \hh^c$, 
along with $x_\alpha\in \g^c_\alpha$ for $\alpha\in \Phi$ such that
\begin{align*}
&[h_i,h_j] = 0\\
&[h_i,x_\alpha] = \langle \alpha,\alpha_i^\vee\rangle  x_\alpha\\
&[x_\alpha,x_{-\alpha}] = h_\alpha\\
&[x_\alpha,x_\beta] = N_{\alpha,\beta} x_{\alpha+\beta},
\end{align*}
where $h_\alpha$ is the unique element in $[\g^c_\alpha,\g^c_{-\alpha}]$ with
$[h_\alpha,x_\alpha]= 2x_\alpha$. This implies that $h_{\alpha_i} = h_i$ for 
$1\leq i\leq \ell$. Furthermore, $N_{\alpha,\beta} = \pm (r+1)$, where
$r$ is the maximal integer with $\alpha-r\beta\in\Phi$. 
Also we define $x_\gamma=0$ if $\gamma\not\in\Phi$. 

A basis with these properties is called a {\em Chevalley basis} of $\g^c$
(see \cite{hum}, \S 25.2).

Let $\imath\in \C$ denote the imaginary unit, and consider the elements
\begin{equation}\label{eq:cpt}
\imath h_1,\ldots, \imath h_\ell \text{ and } x_\alpha-x_{-\alpha},
\imath(x_\alpha+x_{-\alpha}) \text{ for $\alpha\in \Phi^+$}.
\end{equation}
Let $\uu$ denote the $\R$-span of these elements. Then $\uu$ is closed under
the Lie bracket, and hence is a real Lie algebra. This Lie algebra is compact,
and called a {\em compact form} of $\g^c$. We have $\g^c = \uu + \imath \uu$
and we define a corresponding map $\tau : \g^c \to \g^c$ by $\tau(x+\imath y) = 
x-\imath y$, for $x,y\in\uu$. This map is called the {\em conjugation of 
$\g^c$} with respect to $\uu$.

Let $\theta : \g^c\to \g^c$ be an automorphism of order 2, commuting with 
$\tau$. Then $\theta$ maps $\uu$ into itself, and hence $\uu = \uu_1 + 
\uu_{-1}$, where $\uu_k$ denotes the $\theta$-eigenspace with eigenvalue $k$.
Set $\kk = \uu_1$ and $\pp = \imath\uu_{-1}$, and $\g = \kk\oplus \pp$. Then $\g$
is a real subspace of $\g^c$, closed under the Lie bracket. So it is a real 
form of $\g^c$. Also here we get a conjugation, $\sigma : \g^c \to \g^c$,
by $\sigma(x+\imath y) = x-\imath y$ for $x,y\in \g$. The maps $\sigma$,
$\tau$ and $\theta$ pairwise commute, all have order 2 and 
$\tau = \theta\sigma$. 

Every real form of $\g^c$ can be constructed in this way
(see \cite{onishchik}). The decomposition $\g = \kk\oplus \pp$ is called a
{\em Cartan decomposition}. The restriction of $\theta$ to $\g$ is called a 
{\em Cartan involution} of $\g$. 

\subsection{Canonical generators}\label{subsec:pre2}

For $1\leq i\leq \ell$ let $g_i$, $x_i$, $y_i$ be elements of $\g^c$ such that
\begin{equation}\label{eqn:2}
\left. \begin{aligned}
& [g_i,g_j] = 0\\
& [g_i,x_j] = \langle \alpha_j, \alpha_i^\vee\rangle x_j \\
& [g_i,y_j] = -\langle \alpha_j, \alpha_i^\vee\rangle y_j \\
& [x_i,y_j] = \delta_{ij} g_i. 
\end{aligned}\right.
\end{equation}
A set of $3\ell$ elements with these commutation relations is called a 
{\em canonical generating set} of $\g^c$ (\cite{jac}, \S IV.3). 
We have the following:
\begin{itemize}
\item A canonical generating set of $\g^c$ generates $\g^c$.
\item Sending one canonical generating set to another one uniquely extends
to an automorphism of $\g^c$.
\end{itemize}

An example of a canonical generating set is the following: let $g_i=h_i$, $x_i 
= x_{\alpha_i}$, $y_i = x_{-\alpha_i}$ (where we use the notation of Section \ref{subsec:pre1}).

\subsection{Computing endomorphism spaces}\label{subsec:pre3}

Here $\g^c$ is a complex semisimple Lie algebra with canonical generators $h_i$, $x_i$, $y_i$
for $1\leq i\leq \ell$. Let $\hh^c$ denote the span of the $h_i$ (a Cartan subalgebra of $\g^c$).
First we review some of the basic facts of the representation theory
of $\g^c$ (see \cite{hum}, \S 20). 

Let $\rho : \g^c \to \gl(V^c)$ be a finite-dimensional representation of $\g^c$. 
For $\mu \in (\hh^c)^*$ we set $V^c_\mu = \{ v\in V^c \mid \rho(h)v = \mu(h) v\}$.
If $V^c_\mu\neq 0$ then $\mu$ is called a {\em weight} of $\rho$ (or of the $\g^c$-module
$V^c$), and $V^c_\mu$ is the 
corresponding {\em weight space}. Elements of $V^c_\mu$ are called weight vectors of weight $\mu$.
We have that $V^c$ is the sum of its weight spaces. Let $v\in V^c_\mu$ and suppose that 
$\rho(x_i)v=0$
for $1\leq i\leq \ell$. Then $v$ is called a {\em highest weight vector}, and $\mu$ 
a {\em highest weight} of $\rho$. 

Suppose that $\rho$ is irreducible. Then there is a unique highest weight $\lambda$. Moreover,
$\dim V^c_\lambda =1$. Let $v_\lambda\neq 0$ be a highest weight vector of weight $\lambda$.
Then there is a set $S_\lambda$ of sequences $(i_1,\ldots,i_k)$,
with $k\geq 0$ and $1\leq i_r \leq \ell$ such that the elements $\rho(y_{i_1})\cdots \rho(y_{i_k})
v_\lambda$ form a basis of $V^c$. We note that $S_\lambda$ is not uniquely determined. But
for each $\lambda$ we fix one $S_\lambda$. Now let $\varphi : \g^c \to \gl(W^c)$ be another 
irreducible representation of $\g^c$ with the same highest weight $\lambda$. Let $w_\lambda\neq 0$
be a highest weight vector of weight $\lambda$. Define the linear map $A : V^c \to W^c$ that maps
$\rho(y_{i_1})\cdots \rho(y_{i_k})v_\lambda$ to $\varphi(y_{i_1})\cdots \varphi(y_{i_k})w_\lambda$,
for all $(i_1,\ldots,i_k)\in S_\lambda$. 

\begin{lem}\label{lem:rep}
We have $A\rho(x) = \varphi(x)A$ for all $x\in \g^c$.
\end{lem}

\begin{proof}
Since $\rho$, $\varphi$ are irreducible representations of $\g^c$ with the same highest weight,
there exists an isomorphism, that is, a bijective linear map $A' : V^c\to W^c$ with
$A'\rho(x)v = \varphi(x)A'v$ for all $x\in \g^c$ and $v\in V^c$. This implies that 
$A' v_\lambda = a w_\lambda$ where $a\in \C$, $a\neq 0$. It also follows that $A=\tfrac{1}{a} A'$,
whence the statement.
\end{proof}

Now we drop the assumption that $\rho$ is irreducible. Let $\lambda_1,\ldots,\lambda_r$
be the distinct highest weights of $\rho$. For $1\leq j\leq r$ let $v_{j,1},\ldots,v_{j,m_j}$ be
a linearly independent set of highest weight vectors of highest weight $\lambda_j$.
So each $v_{j,l}$ generates an irreducible $\g^c$-submodule,
denoted $V(\lambda_j,l)$, of $V^c$, and $V^c$ is their direct sum. 
We use the basis of $V^c$ consisting of the elements
$\rho(y_{i_1})\cdots \rho(y_{i_k})v_{j,l}$, for $(i_1,\ldots,i_k)\in S_{\lambda_j}$.
For $1\leq j \leq r$ and $1\leq s,t\leq m_j$ we let $A_j^{s,t}$ be the linear map $V^c\to V^c$
that maps $\rho(y_{i_1})\cdots \rho(y_{i_k})v_{j,s}$ to $\rho(y_{i_1})\cdots \rho(y_{i_k})v_{j,t}$
for $(i_1,\ldots,i_k)\in S_{\lambda_j}$, and it maps all other basis elements to 0.
Then $A_j^{s,t}$ is an isomorphism of $V(\lambda_j,s)$ to $V(\lambda_j,t)$, and it maps 
all other submodules $V(\lambda_k,u)$ to $0$. So by Lemma \ref{lem:rep},
$A_j^{s,t}\rho(x) = \rho(x)A_j^{s,t}$ for all $x\in \g^c$, i.e., it is contained in
$$\End_\rho(V^c) = \{ A \in \End(V^c) \mid A\rho(x) = \rho(x)A \text{ for all } x\in \g^c\}.$$

\begin{lem}
The $A_j^{s,t}$ for $1\leq j \leq r$ and $1\leq s,t\leq m_j$ form a basis of $\End_\rho(V^c)$.
\end{lem}

\begin{proof}
Let $A\in \End_\rho(V^c)$. 
Then $A$ is determined by the images $Av_{j,s}$ for $1\leq j\leq r$,
$1\leq s\leq m_j$. But $A$ maps (highest) weight vectors to (highest) weight vectors
of the same weight. So there are $\alpha_j^{s,t}\in\C$ such that
$$Av_{j,s} = \alpha_j^{s,1} v_{j,1} +\cdots + \alpha_j^{s,m_j} v_{j,m_j}.$$
It follows that $A = \sum_{j,s,t} \alpha_j^{s,t} A_j^{s,t}$. It is obvious that the $A_j^{s,t}$ are
linearly independent.
\end{proof}

Now consider a second representation $\varphi : \g^c \to \gl(V^c)$ that is equivalent to $\rho$, 
i.e., there is a bijective linear map $A_0 : V^c\to V^c$ such that $A_0\rho(x) = \varphi(x)A_0$ for
all $x\in \g^c$. In particular, $A_0$ lies in the space
$$\End_{\rho,\varphi}(V^c) = \{ A\in \End(V^c) \mid A\rho(x) = \varphi(x)A \text{ for all }
x\in \g^c\}.$$
We want to find a basis of $\End_{\rho,\varphi}(V^c)$. A first observation is 
that $\End_{\rho,\varphi}(V^c) = \{ A_0A \mid A\in \End_\rho(V^c)\}$. So since above we
have seen how to construct a basis of $\End_\rho(V^c)$, the problem boils down to 
constructing $A_0$. Since $\varphi$ is equivalent to $\rho$ there are $w_{j,1},\ldots,w_{j,m_j}$ 
forming a basis of the weight space with weight $\lambda_j$, relative to the representation
$\varphi$. Applying Lemma \ref{lem:rep} to each submodule $V(\lambda_j,l)$ we see that
mapping $v_{j,l}$ to $w_{j,l}$ (for all $j,l$) uniquely extends to a bijective linear map
$A_0 : V^c\to V^c$, contained in $\End_{\rho,\varphi}(V^c)$.

\subsection{On solving polynomial equations}\label{subsec:pre4}

In the end, the solution to our problem will be given by a set of polynomial equations,
which we need to solve. For this, to the best of our knowledge, no good algorithm is available.
So in each particular case
we have to look at the equations and see whether we can solve them. However, there are
some algorithms that can help with that, most importantly the algorthm for constructing a 
Gr\"obner basis (see \cite{clo}). Let $F$ be a field, and $R= F[x_1,\ldots,x_m]$ the polynomial
ring in $m$ indeterminates over $F$. Let $P\subset R$ be a finite set of polynomials,
and consider the polynomial equations $p=0$ for $p\in P$. We want to determine the set
$V = \{ v\in F^m \mid p(v)=0 \text{ for all } p\in P\}$.
Let $G$ be any other generating set
of the ideal $I$ of $R$ generated by $P$. Then solving $p=0$ for all $p\in P$ is equivalent to 
solving $g=0$ for all $g\in G$ (the set of solutions is the same). A convenient choice
for $G$ is a Gr\"obner basis of $I$ with respect to a lexicographical monomial order. Then $G$ has
a triangular form, which, in most cases, makes solving the equations easier. We refer to
\cite{clo} for a more detailed discussion.

\section{Construction of embeddings}\label{sec:3}

Here we turn to our main problem, stated in Section \ref{sec:1}.

Let $\g^c$, $\tilde\g^c$ be complex semisimple Lie algebras, and suppose that
we have an embedding $\veps : \g^c \hookrightarrow \tilde\g^c$. Let $\hh^c$ be
a fixed Cartan subalgebra of $\g^c$, and let $\Phi$ denote the corresponding
root system. Let $h_1,\ldots,h_\ell$, and $x_\alpha$ for $\alpha\in \Phi$
be a Chevalley basis of $\g^c$. Let $\uu$ be the compact form spanned by the
elements \eqref{eq:cpt}, with corresponding conjugation $\tau$. Let $\g$ be
a real form of $\g^c$ with Cartan decomposition $\g = \kk\oplus \pp$, and
corresponding involution $\theta$, and conjugation $\sigma$. We assume that
$\g$ and $\uu$ are compatible, i.e., $\tau$ and $\sigma$ commute, and 
$\theta = \tau\sigma$ and $\uu = \kk\oplus \imath \pp$. 

\begin{prop}\label{prop:eps}
Let $\tilde\g\subset \tilde\g^c$ be a real form of $\tilde\g^c$ such that 
$\veps(\g)\subset \tilde \g$. Then there are a compact form $\tilde\uu\subset \tilde\g^c$
of $\tilde\g^c$, with conjugation $\tilde\tau : \tilde\g^c\to\tilde\g^c$, and an involution
$\tilde\theta$ of $\tilde\g^c$ such that
\begin{enumerate}
\item \label{eps1}  $\veps(\uu)\subset \tilde \uu$, 
\item \label{eps2} $\veps\theta = \tilde\theta\veps$, 
\item \label{eps3} $\tilde\theta \tilde\tau = \tilde\tau\tilde\theta$, 
\item \label{eps4} there is a Cartan decomposition $\tilde\g = \tilde\kk \oplus \tilde \pp$, 
such that the restriction of $\tilde\theta$ to $\tilde\g$ 
is the corresponding Cartan involution, and
$\tilde\uu = \tilde \kk\oplus \imath \tilde\pp$.
\end{enumerate} 
Conversely, if $\tilde\uu\subset \tilde\g$ is a compact form, with corresponding conjugation
$\tilde\tau$, and $\tilde\theta$ is an involution of $\tilde\g^c$ such that \eqref{eps1}, 
\eqref{eps2} and \eqref{eps3} hold,
then $\tilde\theta$ leaves $\tilde\uu$ invariant, and setting $\tilde\kk = \tilde\uu_1$,
$\tilde\pp = \imath\tilde\uu_{-1}$ (where $\tilde\uu_k$ is the $k$-eigenspace of $\tilde\theta$),
we get that $\tilde\g = \tilde\kk\oplus\tilde\pp$ is a real form of $\tilde\g^c$ with 
$\veps(\g)\subset \tilde\g$. 
\end{prop}

\begin{proof}
There is a Cartan decomposition $\tilde\g = \tilde\kk\oplus\tilde\pp$ such that $\veps(\kk)
\subset \tilde\kk$, $\veps(\pp)\subset\tilde\pp$ (this is the Karpelevich-Mostow theorem, see
\cite{onishchik}, \S 6, Corollary 1).
We let $\tilde\theta$ be the involution
of $\tilde\g^c$ such that $\tilde\theta(x)=x$ for all $x\in \tilde\kk^c$, and $\tilde\theta(x)
=-x$ for all $x\in \tilde\pp^c$. Finally we set $\tilde\uu = \tilde\kk\oplus \imath\tilde\pp$.
Then the statements \eqref{eps1}, \eqref{eps2}, \eqref{eps3}, and \eqref{eps4} are all obvious. 
The converse is clear as well.
\end{proof}

Throughout this section 
let $\tilde\hh^c$ be a fixed Cartan subalgebra of $\tilde\g^c$. 
We let $\Psi$ denote the root system of
$\tilde\g^c$ with respect to $\tilde\hh^c$. 
By $g_1,\ldots,g_m$ together
with $y_\beta$, for $\beta\in\Psi$ we denote a fixed Chevalley basis of 
$\tilde\g^c$. We let $\tilde\uu$ be the compact form of $\tilde\g^c$ spanned
by $\imath g_i$, $1\leq i\leq m$, $y_\beta-y_{-\beta}$,
$\imath(y_\beta+y_{-\beta})$ for $\beta\in\Psi^+$. 

From the formulation of the main problem we see that it does not make a 
difference if we replace $\veps$ by $\phi\veps$, where $\phi\in 
\Aut(\tilde\g^c)$. The first step of our procedure is to replace $\veps$ by a $\phi\veps$
to ensure that $\veps(\uu)\subset \tilde\uu$. This is the subject of Section \ref{sec:emb1}.

In Section \ref{sec:emb2} we show how to find the involutions $\tilde\theta$ with 
Proposition \ref{prop:eps}\eqref{eps2} and \eqref{eps3}. Then Proposition \ref{prop:eps}
shows how to construct the corresponding real forms of $\tilde\g^c$.

We recall (\cite{dyn}, see also \cite{minchenko}, \cite{graaf_sss}) that
two embeddings $\veps, \veps' : \g^c\hookrightarrow \tilde\g^c$ are called
{\em equivalent} if there is an {\em inner} automorphism $\phi$ of $\tilde\g^c$
such that $\veps = \phi \veps'$. They are called {\em linearly equivalent}
if for all representations $\rho : \tilde\g^c \to \gl(V^c)$ the induced
representations $\rho\circ \veps$, $\rho\circ\veps'$ are equivalent. 
Equivalence implies linear equivalence, but the converse is not always true.
However, the cases where the same linear equivalence class splits into more than
one equivalence class are rather rare (cf. \cite{minchenko}, Theorem 7).

\subsection{Embedding the compact form}\label{sec:emb1}

Suppose that $\veps(\hh^c)\subset \tilde\hh^c$. Then 
for $\alpha\in\Phi$ there is a subset $A_\alpha\subset \Psi$ such that
\begin{equation}\label{eqn:3}
\left.\begin{aligned}
\veps(x_\alpha) &= \sum_{\beta\in A_{\alpha}} a_{\alpha,\beta} y_\beta\\
\veps(x_{-\alpha}) &= \sum_{\beta\in A_{\alpha}} b_{\alpha,\beta} y_{-\beta},\\
\end{aligned}\right.
\end{equation}
where $a_{\alpha,\beta}, b_{\alpha,\beta}\in \C$ (in fact, $A_\alpha$ consists
of all $\beta$ which restricted to $\veps(\hh^c)$ equal $\alpha$). 

We say that the embedding $\veps$ is {\em balanced} if $\veps(\hh^c)\subset \tilde\hh^c$ and
for all $\alpha\in \Phi$, and $\beta\in A_\alpha$ we have $b_{\alpha,\beta} = \bar 
a_{\alpha,\beta}$ (complex conjugation). Of course, this notion depends on the choices of
Cartan subalgebras and Chevalley bases in $\g^c$, $\tilde\g^c$. If we use the term ``balanced''
without mentioning these, then we use the choices fixed at the outset. Otherwise
we explicitly mention a different choice made.

\begin{lem}\label{lem:bal}
If $\veps$ is balanced then $\veps(\uu)\subset \tilde\uu$. Conversely, if $\veps(\hh^c)
\subset \tilde\hh^c$ and $\veps(\uu)\subset \tilde\uu$, then $\veps$ is balanced.
\end{lem}

\begin{proof}
By standard arguments one can show that $\veps(h_i)$ is a $\Q$-linear combination of 
the $g_j$. (Set $x=\veps(x_{\alpha_i})$, $y=\veps(x_{-\alpha_i})$, $h=\veps(h_i)$. Then $[x,y]=h$,
$[h,x]=2x$, $[h,y]=-2y$. So by $\ssl_2$-representation theory the eigenvalues of $\ad_{\tilde\g^c}
h$ are integers. 
Let $\{\beta_1,\ldots,\beta_m\}$ be a basis of simple roots of $\Psi$, with
corresponding Cartan matrix $\widetilde{C}$. Then $\beta_j(h)\in \Z$ for all $j$.  
Furthermore, if we write $h = a_1g_1 +\cdots +a_mg_m$,
then we get that the vector $(a_1,\ldots,a_m)$ is $\widetilde{C}^{-1}$ times the vector
$(\beta_1(h),\ldots,\beta_m(h))$. So $a_j\in \Q$.)
In particular, $\veps(\imath h_i)$ lies in the $\R$-span of
$\imath g_1,\ldots,\imath g_m$. 

Also, for $\alpha\in \Phi^+$ we have
\begin{equation}\label{eqn:4}
\left.\begin{aligned}
\veps(x_\alpha-x_{-\alpha}) &= \sum_{\beta\in A_\alpha} a_{\alpha,\beta}y_\beta
-b_{\alpha,\beta} y_{-\beta}\\
&= \sum_{\beta\in A_\alpha} \frac{a_{\alpha,\beta}+b_{\alpha,\beta}}{2}(y_\beta-
y_{-\beta}) - \imath \frac{a_{\alpha,\beta}-b_{\alpha,\beta}}{2}\imath (y_\beta+
y_{-\beta}).
\end{aligned}\right.
\end{equation}
We see that all coefficients lie in $\R$, whence $\veps(x_\alpha-x_{-\alpha})
\in \tilde\uu$. The argument for $\veps(\imath(x_\alpha+x_{-\alpha}))$ is 
enterily similar.

For the converse, from \eqref{eqn:4}
we get that $a_{\alpha,\beta}+b_{\alpha,\beta}\in \R$ and $a_{\alpha,\beta}-
b_{\alpha,\beta}\in \imath \R$. That implies $b_{\alpha,\beta} = \bar 
a_{\alpha,\beta}$.
\end{proof}

The next lemma says that the automorphism that we are after exists. 

\begin{lem}
There exists an inner automorphism $\phi$ of $\tilde\g^c$ such that 
$\phi\veps$ is balanced.
\end{lem}

\begin{proof}
There is a compact form $\tilde\uu'$ of $\tilde\g^c$ such that $\veps(\uu)
\subset \tilde\uu'$ (\cite{onishchik}, \S 6, Proposition 3). 
There is an inner automorphism $\phi'$ of $\tilde\g^c$
such that $\phi'(\tilde\uu') = \tilde\uu$ (\cite{onishchik}, \S 3, Corollary to 
Proposition 6). 
Moreover, the span of the elements
$\phi'(\veps(\imath h_i))$ lies in a Cartan subalgebra of $\tilde\uu$, which
is conjugate to the span of the $\imath g_j$ by an inner automorphism of
$\tilde\uu$. This automorphism extends to an inner automorphism of 
$\tilde\g^c$. So we get an inner automorphism $\phi$ of $\tilde\g^c$ such
that $\phi(\veps(\uu)) \subset \tilde\uu$, and $\phi(\veps(\hh^c))\subset 
\tilde\hh^c$. So by Lemma \ref{lem:bal} we conclude that $\phi\veps$ is balanced.
\end{proof}

Now suppose that $\veps$ has the property that $\veps(\hh^c) \subset 
\tilde\hh^c$, but $\veps$ is not balanced. Let $\Delta=\{\alpha_1,\ldots,
\alpha_\ell\}$ be a fixed basis
of simple roots of $\Phi$. Then we set up a system of polynomial equations.
The indeterminates are $s_{\alpha,\beta}$, $t_{\alpha,\beta}$, where 
$\alpha\in\Delta$, $\beta\in A_\alpha$. For $1\leq i\leq \ell$ we set
\begin{align*}
X_i &= \sum_{\beta\in A_{\alpha_i}}  (s_{\alpha_i,\beta}+\imath t_{\alpha_i,\beta})
y_\beta\\
Y_i &= \sum_{\beta\in A_{\alpha_i}}  (s_{\alpha_i,\beta}-\imath t_{\alpha_i,\beta})
y_{-\beta}\\
\end{align*}

Next we require that the $3\ell$ elements $\veps(h_i)$, $X_i$, $Y_i$ satisfy the
relations \eqref{eqn:2} (where in place of $g_i$ we take $\veps(h_i)$,
in place of $x_i,y_i$ we take $X_i$, $Y_i$). This leads to a set of polynomial
equations in the indeterminates $s_{\alpha,\beta}$, $t_{\alpha,\beta}$, 
which we solve over $\R$. Let $\hat s_{\alpha,\beta}, \hat t_{\alpha,\beta}\in\R$
be the values that we obtain. Let $\widehat{X}_i$, $\widehat{Y}_i$ be the
same as $X_i$, $Y_i$, but with these values substituted. Then mapping 
$h_i$ to $\veps(h_i)$, $x_{\alpha_i}$ to $\widehat{X}_i$, $x_{-\alpha_i}$
to $\widehat{Y}_i$ defines an embedding $\hat \veps : \g^c \to \tilde\g^c$
(see Section \ref{subsec:pre2}). 

\begin{lem}
$\hat\veps$ is balanced.
\end{lem}

\begin{proof}
Consider the elements $x_\alpha-x_{-\alpha}$, $\imath(x_\alpha+x_{-\alpha})$, for $\alpha\in\Delta$
and $\imath h_i$, for $1\leq i\leq \ell$. The span of these over $\C$ is the same as the span
of the canonical generating set consisting of the $x_\alpha$, $x_{-\alpha}$, $h_i$. So they
generate $\g^c$ over $\C$, and since they lie in $\uu$, they generate $\uu$ over $\R$. Moreover,
their images under $\hat\veps$ lie in $\tilde\uu$, so $\hat\veps(\uu)\subset \tilde\uu$. Since
also $\hat\veps(\hh^c) \subset \tilde\hh^c$ we conclude by Lemma \ref{lem:bal}.
\end{proof}

Since $\hat\veps$ agrees with $\veps$ on $\hh^c$, we have that $\veps$ and
$\hat\veps$ are linearly equivalent (see \cite{dyn}, Theorem 1.5, see also
\cite{graaf_sss}, Theorem 4). If the linear equivalence class of $\veps$
does not split into more than one equivalence class, then we are done:
$\veps$ and $\veps'$ are equivalent. If we are in a rare case where there 
are more equivalence classes, then we have to find more solutions to the polynomial equations:
one for each equivalence class contained in the linear equivalence class of $\veps$.

\begin{rem}
For the embeddings that have been determined with the methods of 
\cite{graaf_sss}, the following trick often works. Let $\Pi = \{\beta_1,\ldots,
\beta_m\}$ be a fixed basis of simple roots of $\Psi$. Let $\delta_1,\ldots,
\delta_m\in \C\setminus\{0\}$, and let $\phi$ be the automorphism of 
$\tilde\g^c$ mapping $g_j\mapsto g_j$, $y_{\beta_j}\mapsto \delta_j y_{\beta_j}$
$y_{-\beta_j}\mapsto \delta_j^{-1} y_{-\beta_j}$. Then the images of the
$g_j$, and $y_\beta$ under $\phi$ also form a Chevalley basis of $\tilde\g^c$.
Moreover, $\phi(y_\beta) = \delta_1^{e_1}\cdots \delta_m^{e_m} y_\beta$, if
$\beta = \sum_j e_j \beta_j$. Write $y'_\beta = \phi(y_\beta)= \delta_\beta
y_\beta$. 

Now consider the equations \eqref{eqn:3}, and write $b_{\alpha,\beta}
= \mu_{\alpha,\beta} \bar a_{\alpha,\beta}$.  If we use the basis consisting
of the $y'_\beta$, then we get that the coefficients are $a'_{\alpha,\beta}=
\delta_\beta^{-1} a_{\alpha,\beta}$ and $b'_{\alpha,\beta} = \delta_\beta 
b_{\alpha,\beta}$. So $b'_{\alpha,\beta} = \bar a'_{\alpha,\beta}$ is equivalent to
$\delta_\beta^2 = \mu_{\alpha,\beta}^{-1}$. This then yields a set of 
polynomial equations for the $\delta_i$. It is by no means guaranteed that
this set is consistent (i.e., has any solution at all). However, from our
experience, we get that in many cases the set is not only consistent, but
also a reduced Gr\"obner basis is of the form $\{ \delta_1^2-r_1,\ldots,
\delta_m^2-r_m\}$, with $r_i\in \R$, $r_i>0$, which makes solving the equations
extremely easy. 

A solution of the equations yields an automorphism $\phi$ of $\tilde\g^c$
such that $\phi(\tilde\uu)=\tilde\uu'$, where $\tilde\uu'$ is the compact form
spanned by the elements $\imath g_j$, $y'_\beta-y'_{-\beta}$, 
$\imath (y'_\beta+y'_{-\beta})$. Moreover, $\veps$ is balanced with respect to
the Chevalley basis consisting of the $y'_\beta$, so that
$\veps(\uu)\subset \tilde\uu'$. So if we set $\veps' = \phi^{-1}\veps$, then
$\veps'$ is equivalent to $\veps$ and $\veps'(\uu) \subset \tilde\uu$.
\end{rem} 

\subsection{Finding $\tilde\theta$}\label{sec:emb2}

Here we assume that we have an embedding $\veps : \g^c \hookrightarrow 
\tilde\g^c$ such 
that $\veps(\hh^c) \subset \tilde\hh^c$ and $\veps(\uu) \subset \tilde\uu$.
Now we focus on the problem of finding the involutions $\tilde\theta$ of
$\tilde\g^c$ such that $\veps\theta = \tilde\theta\veps$.

Let $\ad : \tilde\g^c\to 
\gl(\tilde\g^c)$ be the adjoint representation, i.e., $\ad x(y) = [x,y]$. 
Set
$$\A = \{ A\in \End(\tilde\g^c) \mid
A\ad(\veps\theta(y)) = \ad(\veps(y))A \text{ for all } y \in \g^c\}.$$

\begin{prop}\label{lem:theta}
Let $\tilde\theta\in \End(\tilde\g^c)$. Then $\tilde\theta$ is an involution of 
$\tilde\g^c$ with $\veps\theta = \tilde\theta\veps$ if and only if $\tilde\theta\in\A$
and
\begin{enumerate}
\item \label{theta1} $\tilde\theta^2 = I$, where $I\in \End(\tilde\g^c)$ is the identity,
\item \label{theta2} $\tilde\theta (\ad x)\tilde\theta = \ad \tilde\theta(x)$ for all 
$x\in \tilde\g^c$.
\end{enumerate}
\end{prop}

\begin{proof}
Suppose that $\tilde\theta$ is an involution of $\tilde\g^c$. Then \eqref{theta1} is immediate.
Also for $y\in \tilde\g^c$ we have $\tilde\theta (\ad x)\tilde\theta (y) = \tilde\theta[x,
\tilde\theta(y)] = \ad \tilde\theta(x)(y)$, so \eqref{theta2} follows. Together with
$\veps\theta = \tilde\theta\veps$ this also implies that $\tilde\theta\in\A$.

For the converse we first show that $\tilde\theta$ is an involution of $\tilde\g^c$.
From \eqref{theta1} it follows that it is bijective and that it has order 2. Using 
\eqref{theta2} we get
$\tilde\theta [x,y] = \tilde\theta \ad x (y) = \ad \tilde\theta x (\tilde\theta y) = 
[\tilde\theta(x),\tilde\theta(y)]$.
Secondly, $\tilde\theta\veps= \veps\theta$ is equivalent to $\ad \tilde\theta\veps(y)= 
\ad \veps\theta(y)$ for all $y\in \g^c$. Using \eqref{theta1} and \eqref{theta2} 
it is straightforward to see that
this is the same as $\tilde\theta\in \A$.
\end{proof}


We let $a_1,\ldots,a_n$ be a fixed basis of $\tilde\g^c$ (for example, the Chevalley basis
fixed at the start). The idea now is to translate the conditions of Proposition \ref{lem:theta}
into polynomial equations. For that we proceed as follows:

\begin{enumerate}
\item Compute a basis $A_1,\ldots,A_s$ of $\A$ (see Section \ref{subsec:pre3}; note that,
if we let $\rho, \varphi : \g^c \to \gl(\tilde\g^c)$ be the representations given by 
$\rho(y) = \ad \veps\theta(y)$, $\varphi(y) = \ad \veps(y)$, then
$\A = \End_{\rho,\varphi}(\tilde\g^c)$).
\item Let $z_1,\ldots,z_s$ be indeterminates over $\C$, and set
$A = z_1A_1+\cdots +z_sA_s$. Then $A^2=I$ is equivalent to a set
of polynomial equations in the $z_i$. Let $P_1$ denote the corresponding set of polynomials.
\item We note that Proposition \ref{lem:theta}\eqref{theta2} is equivalent to 
$A\ad a_j A = \ad Aa_j$ for $1\leq j\leq n$. Also this is equivalent to a set of polynomial
equations in the $z_i$. Let $P_2$ denote the corresponding set of polynomials.
\end{enumerate}

Now we consider the compact form $\tilde\uu$, and the corresponding conjugation
$\tilde\tau : \tilde\g^c \to \tilde\g^c$. We want to construct involutions $\tilde\theta$
of $\tilde\g^c$ that commute with $\tilde\tau$ (or, equivalently, that leave $\tilde\uu$
invariant). First we observe that it is straightforward to compute
$\tilde\tau(x)$ for an $x\in \tilde\g^c$. Indeed, let $u_1,\ldots,u_n$ be a basis of $\tilde\uu$,
and write $x = \sum_i \alpha_i u_i$, with $u_i\in \C$. Then $\tilde\tau(x) = \sum_i \bar\alpha_i
u_i$. 

Let $R = \R[x_1,\ldots,x_s, y_1,\ldots,y_s]$. We substitute $x_i +\imath y_i$
for $z_i$ in the polynomials in the sets $P_1$, $P_2$. A polynomial $f$ in one of these sets
then transforms into $g+\imath h$, with $g,h\in R$. The polynomial equation $f=0$ is equivalent
to two polynomial equations, this time over $\R$, $g=h=0$. This way we obtain a set of polynomials
$Q_1 \subset R$. 

Let $A = \sum_{i=1}^s (x_i +\imath y_i)A_i$, then $\tilde\tau A(a_j) = A\tilde\tau(a_j)$ is the same
as 
$$\sum_{i=1}^n (x_i -\imath y_i) \tilde\tau(A_ia_j) = \sum_{i=1}^n (x_i +\imath y_i)A_i
\tilde\tau(a_j).$$
Again we split the real and imaginary parts. Doing this for $1\leq j\leq n$ we
obtain a system of (linear) polynomial equations. The corresponding set of polynomials is 
denoted by $Q_2$. 

Finally we solve the system of polynomial equations $q=0$ for $q\in Q_1\cup Q_2$.
Let $\tilde\g_1,\ldots,\tilde\g_m$ be fixed noncompact real forms of $\tilde\g^c$, such that
each noncompact real form of $\tilde\g^c$ is isomorphic to exactly one of the
$\tilde\g_i$. Each solution of the polynomial equations yields an involution $\tilde\theta$
of $\tilde\g^c$, and we construct the corresponding real form $\tilde\g$ as in Proposition
\ref{prop:eps}. The using the methods of \cite{dfg} we find an isomorphism $\tilde\g \to 
\tilde\g_i$, and hence we can map $\g$ to a subalgebra of an appropriate $\tilde\g_i$. 

\begin{rem}
This method works best when the polynomial equations have a finite set of solutions: we
list them all, and obtain all $\tilde\g_i$ such that $\g$ maps to a subalgebra by an automorphism
of $\tilde\g^c$. However, it can happen that the set of solutions is infinite. Example
\ref{exa:1} describes a situation where we can deal with that.
\end{rem}

\section{Implementation and examples}\label{sec:4}

As stated in the introduction, we have implemented the algorithms described here 
in the computer algebra system {\sf GAP}4, using the package {\sf CoReLG}.
The main bottleneck of the method is the need to solve a system of polynomial
equations. One of the main parameters influencing the complexity of this system is the
dimension of the space $\A$, since the number of indeterminates is $2\dim \A$. 
(Although, of course, there are also some linear equations, effectively reducing the
number of indeterminates.) From Section \ref{subsec:pre3} we see that $\dim\A =
\sum_{i=1}^r m_i^2$, where the $m_i$ are the multiplicities of the irreducible 
$\g^c$-submodules of $\tilde\g^c$. It can happen that $\dim \A$ is so large that the 
polynomial equations become unwieldy. For example, if $\veps(\g^c)$ is the regular
subalgebra of type $A_1+A_1$ of $F_4$, then $\dim \A = 159$. On the other hand, there are
many subalgebras that lead to equations systems that we can deal with. In this section
we give some examples. An especially favourable situation arises when $\veps(\g^c)$ is an
$S$-subalgebra. That will be the subject of the next section.

In the last two examples we also report on the running times. They have been obtained
on a 3.16 GHz processor. We remark here that there are two fundamental inefficiencies
affecting these running times: firstly, we work over a field containing the square root of all
integers. This field has been implemented by ourselves in {\sf GAP} (see \cite{dg}); however,
since there is no {\sf GAP} kernel support for it, computations using this field tend to take
markedly longer that, say, over $\Q$. Secondly, we create a lot of polynomials, and also the
polynomial arithmetic in {\sf GAP} is not the most efficient possible (essentially for the
same reason as for our field). 

\begin{exa}\label{exa:1}
Let $\tilde\g^c$, $\g^c$ be the Lie algebras of type $A_3$ and $A_2$ respectively. We consider 
the simplest possible embedding: Let $\alpha_1,\alpha_2,\alpha_3$ denote the simple roots
of the root system of $\tilde\g^c$, ordered as usual; then the subalgebra generated by 
$x_{\alpha_i}$, $x_{-\alpha_i}$ for $i=1,2$ is isomorphic to $\g^c$. We consider the real form
of $\g^c$ isomorphic to $\ssl_3(\R)$ (i.e., the split form).

Since the image of $\g^c$ in $\tilde\g^c$ is regular, i.e., is generated by root vectors
of $\tilde\g^c$, it is automatic that $\veps(\uu)\subset \tilde\uu$.

In this case $\A$ has dimension 4. We get a set of 46 polynomial equations in the unknowns
$x_i$, $y_i$, $1\leq i\leq 4$. The reduced Gr\"obner basis of the ideal generated by these
polynomials is 
$$\{x_1-1, x_2-x_3, x_3^2+y_3^2-1, x_4+1, y_1, y_2+y_3, y_4\}.$$
So there is an infinite number of solutions. Now we set $z_1=1$, $z_2=x_3-\imath y_3$, 
$z_3 = x_3+\imath y_3$, $z_4=-1$ (i.e., we work symbolically with $x_3$, $y_3$) and
$A = z_1A_1+\cdots +z_4A_4$. Then the characteristic polynomial of $A$ is 
$$ T^{15} + 3T^{14} + (-3x_3^2 - 3y_3^2)T^{13}+\cdots + 
(3x_3^6 + 9x_3^4y_3^2 + 9x_3^2y_3^4 + 3y_3^6)T+
x_3^6 + 3x_3^4y_3^2 + 3x_3^2y_3^4 + y_3^6.$$
However, using $x_3^2+y_3^2=1$, this reduces to
\begin{multline*} 
T^{15} + 3T^{14} - 3T^{13} - 17T^{12} - 3T^{11} + 39T^{10} + 25T^9 - 45T^8 - 45T^7 + 25T^6 + \\
39T^5 - 3T^4 - 17T^3 - 3T^2+ 3T + 1
\end{multline*}
which is $(T-1)^6(T+1)^9$. From this we conclude that if we take any solution of the equations
and construct the corresponding real form $\tilde\g$, then its Cartan decomposition will
be $\tilde\g = \tilde\kk\oplus \tilde\pp$ with $\dim\tilde\kk = 6$ and $\dim\tilde\pp = 9$.
Now there is, up to isomorphism, only one real form of $\tilde\g^c$ with a Cartan decomposition
satisfying this, namely $\ssl_4(\R)$. Also, up to equivalence, $\tilde\g^c$ contains exactly
one subalgebra isomorphic to $\g^c$. So
we conclude that $\ssl_4(\R)$ is the only real form of
$\tilde\g^c$ containing a subalgebra isomorphic to $\ssl_3(\R)$. 
\end{exa}

\begin{exa}\label{exa:2}
Let $\tilde\g^c$, $\g^c$ be the Lie algebras of type $E_8$ and $A_1+G_2+G_2$ respectively.
As real form $\g$ we took the direct sum of the noncompact real forms of $A_1$ and $G_2$ (twice)
respectively.
In this case $\A$ was computed in 2058 seconds, and $\dim \A = 6$. The polynomial equations
were computed in 36783 seconds. The set $Q_1\cup Q_2$ contains 37460 polynomials. However,
a reduced Gr\"obner basis of the ideal generated by them is
$$ \{x_1+1,x_2,x_3-1,x_4+1,x_5-1,x_6+1,y_1,y_2,y_3,y_4,y_5,y_6\}.$$
So there is only one solution. The corresponding real form of $E_8$ turned out to be $E$VIII. 
\end{exa}

\begin{exa}\label{exa:3}
Let $\tilde\g^c$ be of type $E_6$. Then, up to equivalence, $\tilde\g^c$ contains a unique
subalgebra of type $B_4$. So let $\g^c$ be of type $B_4$ and let $\g = \so(4,5)$. 
In this example $\A$ was computed in 55 seconds, and $\dim\A = 7$. The polynomial equations were 
computed in 510 seconds, the reduced Gr\"obner basis of the ideal generated by them is 
\begin{multline*}
\{x_5^2-x_7, x_5x_6, x_6^2+y_6^2+x_7-1, x_5x_7-x_5, x_6x_7, x_7^2-x_7, x_5y_6, x_7y_6, x_1+x_5, 
x_2+x_6, \\
x_3+1, x_4+x_7, y_1, y_2-y_6, y_3, y_4, y_5, y_7\}.\end{multline*}
We see that $x_7$ can have the values 0,1. Adding $x_7$ to the generating set, the Gr\"obner
basis becomes
$$\{ x_6^2+y_6^2-1, x_1, x_2+x_6, x_3+1, x_4, x_5, x_7, y_1, y_2-y_6, y_3, y_4, y_5, y_7 \}.$$
Here the value of $x_6,y_6$ determines the solution completely. Furthermore, there is an infinite
number of possible values for those indeterminates. However, with the same method as in 
Example \ref{exa:1}, we established that all solutions lead to the inclusion $\so(4,5)\subset 
E\mathrm{I}$. 

Adding $x_7-1$ to the generating set, we get the Gr\"obner basis
$$\{ x_5^2-1, x_1+x_5, x_2, x_3+1, x_4+1, x_6, x_7-1, y_1, y_2, y_3, y_4, y_5, y_6, y_7 \}.$$
Here we get two solutions, which both yield the inclusion $\so(4,5)\subset 
E\mathrm{II}$. 
\end{exa}

\section{$S$-subalgebras of the exceptional Lie algebras}\label{sec:Ssub}

In this section we consider embeddings $\veps : \g^c\hookrightarrow \tilde\g^c$,
such that $\veps(\g^c)$ is a maximal $S$-subalgebra of $\tilde\g^c$, and the latter is of
exceptional type. 

Let $\g$ be a real form of $\g^c$. By \cite{onishchik}, \S 6, Theorem 2, if $\veps(\g^c)$ is an 
$S$-subalgebra of $\tilde\g^c$, then there are 
at most two real forms of $\tilde\g^c$ that contain $\veps(\g)$. And if $\tilde\g^c$ has
no outer automorphisms there is at most one such real form. This explains why our method works
particularly well in this case: the polynomial equations have at most two solutions.
Example \ref{exa:2} illustrates this phenomenon (there the subalgebra is a non maximal
$S$-subalgebra). 

Table \ref{tab:1} contains the results that we obtained using our programs (for the situation
described above, i.e., $\veps(\g^c)$ is a maximal $S$-subalgebra of $\tilde\g^c$). 
We describe the subalgebras
of the complex Lie algebras by giving the type of their root systems, with an upper index
denoting the Dynkin index (see \cite{dyn}).

Komrakov (\cite{komrakov}) has also published a list of the $S$-subalgebras of the real simple
Lie algebras of exceptional type. In type $E_6$ we find a few differences: the inclusions
marked by a $(*)$ are not contained in Komrakov's list. About all other inclusions Komrakov's
list and ours agree.

\begin{longtable}{|l|c|}
\caption{Maximal $S$-subalgebras of the real Lie algebras of exceptional type.}\label{tab:1}
\endfirsthead
\hline
\multicolumn{2}{|l|}{\small\slshape $S$-subalgebras} \\
\hline
\endhead
\hline
\endfoot
\endlastfoot

\hline
complex inclusion & real inclusion\\
\hline

$A_2^9 \subset E_6$ & $\left\{
\begin{aligned}
      \su(1,2) \subset E\mathrm{ II}\\
      \ssl(3, \mathbb{R}) \subset E\mathrm{ II}
 \end{aligned} \right. $ \\ 
$G_2^3 \subset E_6$ & $G \subset E\mathrm{ II}$ (*)\\
$A_2^2 \oplus G_2^1 \subset E_6$ & $\left\{
    \begin{aligned}
      \su(3) \oplus G^{cmp} \subset E\mathrm{ I} \\ 
      \su(1,2) \oplus G \subset E\mathrm{ III}\\
      \su(1,2) \oplus G^{cmp} \subset E\mathrm{ II} (*)\\
      \ssl(3,\mathbb{R}) \oplus G \subset E\mathrm{ IV}\\
      \ssl(3,\mathbb{R}) \oplus G^{cmp} \subset E\mathrm{ I} (*)
    \end{aligned} \right.$\\

$C_4^1 \subset E_6$ & $\left\{
    \begin{aligned}
      \ssp(2,2) \subset E\mathrm{ II} (*)\\
      \ssp(2,2) \subset E\mathrm{ IV} (*)\\
      \ssp(1,3) \subset E\mathrm{ III} (*)\\
      \ssp(1,3) \subset E\mathrm{ I} (*)\\
      \ssp(4,\mathbb{R}) \subset E\mathrm{ II} (*)\\
      \ssp(4,\mathbb{R}) \subset E\mathrm{ I} (*)\\
    \end{aligned}\right.$\\
 $F_4^1 \subset E_6$ & $\left\{
    \begin{aligned}
      F\mathrm{I} \subset E\mathrm{ I} (*) \\ 
      F\mathrm{II} \subset E\mathrm{ III} (*)
    \end{aligned}\right.$ \\
\hline
$A_1^{231}\subset E_7$  & $\ssl(2, \mathbb{R}) \subset E\mathrm{ V}$ \\
$A_1^{399}\subset E_7$  & $\ssl(2, \mathbb{R}) \subset E\mathrm{ V}$ \\
$A_2^{21} \subset E_7$ & $\left\{
   \begin{aligned}
    \su(1,2) \subset E\mathrm{ VI} \\
    \ssl(3, \mathbb{R}) \subset E\mathrm{ V} 
   \end{aligned}\right.$\\
 $A_1^{15} \oplus A_1^{24} \subset E_7$ & $\left\{ 
   \begin{aligned}
    \su( 2 )\oplus \ssl(2,\mathbb{R}) \subset E\mathrm{ V}\\
    \ssl(2,\mathbb{R})\oplus \su( 2 )\subset E\mathrm{ VI}\\
    \ssl(2,\mathbb{R})\oplus \ssl(2,\mathbb{R}) \subset E\mathrm{ VI}
   \end{aligned}\right.$\\
  $A_1^{7} \oplus G_2^2 \subset E_7$ & $\left\{
   \begin{aligned}
    \su( 2 )\oplus G \subset E\mathrm{ VI}\\
    \ssl(2,\mathbb{R})\oplus G^{cmp} \subset E\mathrm{ V}\\
    \ssl(2,\mathbb{R})\oplus G \subset E\mathrm{ V}
   \end{aligned} \right.$\\
 $C_3^{1} \oplus G_2^1 \subset E_7$ & $\left\{
   \begin{aligned}
    \ssp(3)\oplus G \subset E\mathrm{ VI}\\
    \ssp(1,2)\oplus G^{cmp} \subset E\mathrm{ VI}\\
    \ssp(1,2)\oplus G \subset E\mathrm{ VI}\\
    \ssp(3,\mathbb{R})\oplus G^{cmp} \subset E\mathrm{ VII}\\
    \ssp(3,\mathbb{R})\oplus G \subset E\mathrm{ V}
   \end{aligned}\right.$\\
  $A_1^{3} \oplus F_4^1 \subset E_7$ & $\left\{ 
   \begin{aligned}
    \su(2)\oplus F\mathrm{I} \subset E\mathrm{ VI}\\
    \su(2)\oplus F\mathrm{II}  \subset E\mathrm{ VI}\\
    \ssl(2,\R)\oplus F_4^{cmp} \subset E\mathrm{ VII}\\
    \ssl(2,\R)\oplus F\mathrm{I} \subset E\mathrm{ V}\\
    \ssl(2,\R)\oplus F\mathrm{II}  \subset E\mathrm{ VII}
   \end{aligned} \right.$\\
\hline
 $A_1^{520} \subset E_8$ & $\ssl(2,\R)\subset E\mathrm{ VIII}$ \\
 $A_1^{760} \subset E_8$ & $\ssl(2,\R)\subset E\mathrm{ VIII}$ \\
 $A_1^{1240} \subset E_8$ & $\ssl(2,\R)\subset E\mathrm{ VIII}$ \\
  $B_2^{120} \subset E_8$ & $\left\{
   \begin{aligned}
    \so(2,3) \subset E\mathrm{ VIII}\\
    \so(4,1) \subset E\mathrm{ VIII}
   \end{aligned} \right. $\\
  $A_1^{16} \oplus A_2^{6} \subset E_8$ & $\left\{
   \begin{aligned}
    \su(2) \oplus \su(1,2) \subset E\mathrm{ VIII}\\
    \su(2) \oplus \ssl(3,\mathbb{R}) \subset E\mathrm{IX}\\
    \ssl(2,\mathbb{R}) \oplus \su(3) \subset E\mathrm{ VIII}\\ 
    \ssl(2,\mathbb{R}) \oplus \su(1,2) \subset E\mathrm{ VIII}\\ 
    \ssl(2,\mathbb{R}) \oplus \ssl(3,\mathbb{R}) \subset E\mathrm{ VIII}
   \end{aligned} \right.$\\
 $F_4^1 \oplus G_2^1 \subset E_8$ & $\left\{
   \begin{aligned}
    F_4^{cmp} \oplus G \subset E\mathrm{ IX}\\
    F\mathrm{I} \oplus G^{cmp} \subset E\mathrm{ IX}\\ 
    F\mathrm{I} \oplus G \subset E\mathrm{ VIII}\\
    F\mathrm{II} \oplus G^{cmp} \subset E\mathrm{ VIII}\\
    F\mathrm{II} \oplus G \subset E\mathrm{ IX}
   \end{aligned}\right. $\\
\hline
  $A_1^{156} \subset F_4$ & $\ssl(2,\R) \subset F\mathrm{ I}$\\
  $A_1^{8} \oplus G_2^1 \subset F_4$ & $\left\{
   \begin{aligned}
    \su(2)\oplus G \subset F\mathrm{I}\\
    \ssl(2,\R)\oplus G^{cmp} \subset F\mathrm{ II}\\
    \ssl(2,\R)\oplus G \subset F\mathrm{ I}\\
   \end{aligned}\right . $\\
\hline
$A_1^{28} \subset G_2$ & $\ssl(2,\R) \subset G$\\
\hline  
\end{longtable}

\def\cprime{$'$} \def\cprime{$'$} \def\Dbar{\leavevmode\lower.6ex\hbox to
  0pt{\hskip-.23ex \accent"16\hss}D} \def\cprime{$'$} \def\cprime{$'$}
  \def\cprime{$'$} \def\cprime{$'$} \def\cprime{$'$}


\begin{thebibliography}{10}

\bibitem{cornsub1}
J.~F. Cornwell.
\newblock Semi-simple real subalgebras of non-compact semi-simple real {L}ie
  algebras. {I}, {II}.
\newblock {\em Rep. Mathematical Phys.}, 2(4):239--261; ibid. 2 (1971), no. 4,
  289--309, 1971.

\bibitem{cornsub2}
J.~F. Cornwell.
\newblock Semi-simple real subalgebras of non-compact semi-simple real {L}ie
  algebras. {III}.
\newblock {\em Rep. Mathematical Phys.}, 3(2):91--107, 1972.

\bibitem{clo}
D.~Cox, J.~Little, and D.~O'Shea.
\newblock {\em Ideals, Varieties and Algorithms: An Introduction to
  Computational Algebraic Geometry and Commutative Algebra}.
\newblock Springer Verlag, New York, Heidelberg, Berlin, 1992.

\bibitem{corelg}
Heiko Dietrich, Paolo Faccin, and Graaf.
\newblock {\sf CoReLG}, {C}omputation with {R}eal {L}ie {G}roups.
\newblock A {\sf GAP}4 package, 2013.
\newblock in preparation, \verb+(http://science.unitn.it/~corelg/index.html)+.

\bibitem{dfg}
Heiko Dietrich, Paolo Faccin, and Willem A.~de Graaf.
\newblock Computing with real {L}ie algebras: real forms, {C}artan
  decompositions, and {C}artan subalgebras.
\newblock {\em J. Symbolic Comput.}, 56:27--45, 2013.

\bibitem{dg}
Heiko Dietrich and Willem A.~de Graaf.
\newblock A computational approach to the kostant-sekiguchi correspondence.
\newblock {\em Pacific Journal of Mathematics}, 265(2):349--379, 2013.

\bibitem{dyn0}
E.~B. Dynkin.
\newblock Maximal subgroups of the classical groups.
\newblock {\em Trudy Moskov. Mat. Ob\v s\v c.}, 1:39--166, 1952.
\newblock English translation in: Amer. Math. Soc. Transl. (6), (1957),
  245--378.

\bibitem{dyn}
E.~B. Dynkin.
\newblock Semisimple subalgebras of semisimple {L}ie algebras.
\newblock {\em Mat. Sbornik N.S.}, 30(72):349--462 (3 plates), 1952.
\newblock English translation in: Amer. Math. Soc. Transl. (6), (1957),
  111--244.

\bibitem{cornsub3}
J.~M. Ekins and J.~F. Cornwell.
\newblock Semi-simple real subalgebras of non-compact semi-simple real {L}ie
  algebras. {IV}.
\newblock {\em Rep. Mathematical Phys.}, 5(1):17--49, 1974.

\bibitem{cornsub4}
J.~M. Ekins and J.~F. Cornwell.
\newblock Semi-simple real subalgebras of non-compact semi-simple real {L}ie
  algebras. {V}.
\newblock {\em Rep. Mathematical Phys.}, 7(2):167--203, 1975.

\bibitem{gap4}
The GAP~Group.
\newblock {\em {GAP -- Groups, Algorithms, and Programming, Version 4.5}},
  2012.
\newblock \verb+(http://www.gap-system.org)+.

\bibitem{graaf_sss}
Willem A.~de Graaf.
\newblock Constructing semisimple subalgebras of semisimple lie algebras.
\newblock {\em J. Algebra}, 325(1):416--430, 2011.

\bibitem{sla}
Willem A.~de Graaf.
\newblock {\sf SLA} - computing with {S}imple {L}ie {A}lgebras.
\newblock a {\sf GAP} package, 2013.
\newblock \verb+(http://science.unitn.it/~degraaf/sla.html)+, version 0.13.

\bibitem{hum}
J.~E. Humphreys.
\newblock {\em {Introduction to Lie Algebras and Representation Theory}}.
\newblock Springer Verlag, New York, Heidelberg, Berlin, 1972.

\bibitem{jac}
N.~Jacobson.
\newblock {\em {Lie Algebras}}.
\newblock Dover, New York, 1979.

\bibitem{knapp02}
A.~W. Knapp.
\newblock {\em Lie groups beyond an introduction}, volume 140 of {\em Progress
  in Mathematics}.
\newblock Birkh\"auser Boston Inc., Boston, MA, second edition, 2002.

\bibitem{komrakov}
B.~P. Komrakov.
\newblock Maximal subalgebras of real {L}ie algebras and a problem of {S}ophus
  {L}ie.
\newblock {\em Dokl. Akad. Nauk SSSR}, 311(3):528--532, 1990.

\bibitem{minchenko}
A.~N. Minchenko.
\newblock Semisimple subalgebras of exceptional {L}ie algebras.
\newblock {\em Tr. Mosk. Mat. Obs.}, 67:256--293, 2006.
\newblock English translation in: Trans. Moscow Math. Soc. 2006, 225--259.

\bibitem{onishchik}
Arkady~L. Onishchik.
\newblock {\em Lectures on Real Semisimple {L}ie Algebras and Their
  Representations}.
\newblock European Mathematical Society, Z\"urich, 2004.

\end{thebibliography}
\end{document}